\documentclass[12pt]{amsart}
\usepackage{amssymb,latexsym, amsmath, amsxtra}
\usepackage{graphicx}
\newtheorem{conjecture}{Conjecture}
\newtheorem{theorem}{Theorem}
\newtheorem{lemma}{Lemma}

\begin{document}

\title{Moments of zeta and correlations of divisor-sums: IV}
 
\author{Brian Conrey}
\address{American Institute of Mathematics, 600 E. Brokaw Rd., San Jose, CA 95112, USA and School of Mathematics, University of Bristol, Bristol BS8 1TW, UK}
\email{conrey@aimath.org}
\author{Jonathan P. Keating}
\address{School of Mathematics, University of Bristol, Bristol BS8 1TW, UK}
\email{j.p.keating@bristol.ac.uk}

\thanks{We gratefully acknowledge support under EPSRC Programme Grant EP/K034383/1
LMF: L-Functions and Modular Forms.  Research of the first author was also supported 
by the American Institute of Mathematics and by a grant from the National Science 
Foundation. JPK is grateful for the following additional support: a grant from the 
Leverhulme Trust, a Royal Society Wolfson Research 
Merit Award, a Royal Society Leverhulme Senior Research Fellowship, and a grant from the Air Force Office of Scientific Research, Air Force Material Command, 
USAF (number FA8655-10-1-3088). He is also pleased to thank the American Institute of Mathematics for hospitality during a visit where this work started.}

\date{\today}

\begin{abstract} In this series 
we examine the calculation of the $2k$th  moment and shifted moments of the Riemann 
zeta-function on the critical line using long Dirichlet 
polynomials and divisor correlations.   The present paper begins the general study 
of what we call Type II sums which utilize a circle method framework 
and a convolution of shifted convolution sums
  to obtain all of the lower order terms in the asymptotic 
formula for the mean square along
 $[T,2T]$ of a Dirichlet polynomial of length 
up to $T^3$ with divisor functions as coefficients.
\end{abstract}

\maketitle

\section{Introduction}
This paper is part 4 of a sequence of papers devoted to understanding how to conjecture all of 
the integral moments of the Riemann zeta-function from a number theoretic perspective. The method is to approximate
$\zeta(s)^k$ by a long Dirichlet polynomial and then compute the mean square of the Dirichlet polynomial (c.f.~[GG]). 
There will be many off-diagonal terms and it is the care of these that is the concern of these papers. In
particular it is necessary to treat the off-diagonal terms by a method invented by Bogomolny and Keating [BK1, BK2]. 
Our perspective on
this method is that it is most properly viewed as a multi-dimensional Hardy-Littlewood circle method. 

In part 3 [CK3] we considered the type I off diagonal terms from a general perspective. 
Now we look at the simplest type II sums.

The  formula we obtain is in complete agreement with all of the main terms  
predicted by the recipe of [CFKRS] (and in particular, with the leading order term conjectured in [KS]). 
 
\section{Shifted moments}
We are interested  in developing a number theoretic approach  
to the moments of the Riemann zeta-function on the critical line, in particular to the general ``shifted'' moment
given by
\begin{eqnarray}
\label{eqn:moment} 
I^\psi_{A,B}(T)=\int_0^\infty \psi\left(\frac t T \right) 
 \prod_{\alpha\in A}\zeta(s+\alpha)\prod_{\beta\in B} \zeta(1-s+\beta) ~dt
\end{eqnarray}
where $\psi$ is a smooth function with compact support, say $\psi \in C^\infty[1,2]$ and $s=1/2+it$
and $A$ and $B$ are sets of small complex numbers, referred to as the shifts. 
It is useful to consider as well the general shifted moment of a long Dirichlet polynomial. To express this
we first introduce the generalized divisor function $\tau_A(n)$ by way  of its generating function: 
$$\prod_{\alpha\in A} \zeta(s+\alpha)=\sum_{n=1}^\infty \frac{\tau_A(n)}{n^s}=:D_A(s).$$
Then we let 
$$\mathcal D_A(s;X)=\sum_{n\le X} \frac{\tau_A(n)}{n^s}$$
and consider
\begin{eqnarray} \label{eqn:basicI}  \nonumber
I^\psi_{A,B}(T;X):&=&\int_0^\infty \psi\left(\frac t T\right) D_{A}(s;X)D_{B}(1-s;X) ~dt\\
&=& T\sum_{m,n\le X }\frac{\tau_A(m)\tau_B(n){\hat \psi}\left(\frac{T}{2\pi}\log \frac m n \right)}{\sqrt{mn}}.
\end{eqnarray}

The recipe [CFKRS] tells us how to predict the behaviour of these moments. Firstly, we conjecture that 
$$ I^\psi_{A,B}(T)=T\int_0^\infty \psi(t)\sum_{U\subset A, V\subset B\atop |U|=|V|}
 \left(\frac{tT}{2\pi}\right)^{-\sum_{{\alpha}\in U\atop
{\beta}\in V}({\alpha}+{\beta}) }\mathcal B(A-U+V^-,B-V+U^-)~dt+o(T)
$$
where $\mathcal B$ is given by 
$$\mathcal B(A,B)=\sum_{n=1}^\infty\frac {\tau_A(n)\tau_B(n)}{n}$$
in the case that this series converges (for example if $\Re \alpha, \Re \beta>0$ for all $\alpha\in A$ and
$\beta \in B$) and is given by analytic continuation otherwise.
An alternate expression is $\mathcal B(A,B)=\mathcal A(A,B)\mathcal Z(A,B)$ 
where
$$Z(A,B):=\prod_{\alpha\in A \atop \beta\in B} \zeta(1+\alpha+\beta)$$
and $\mathcal A (A,B)$ is a product over primes that converges nicely
in the domains under consideration (see below).  We have used an unconventional notation here; 
 by $A-U+V^-$ we mean the following: start with the set $A$ and remove the elements of $U$ and then include the negatives of 
the elements of $V$. We think of the process as ``swapping" equal numbers of elements between  $A$ and $B$; when elements are removed from $A$
and put into $B$ they first get multiplied by $-1$. We keep track of these swaps with our equal-sized subsets $U$ and $V$  of $A$ and $B$;
and when we refer to the ``number of swaps'' in a term we mean the cardinality $|U|$ of $U$ (or, since they are of equal size, of $V$). 

The Euler product $\mathcal A$ is given by
 \begin{eqnarray*}\label{eq:Azeta}
  \mathcal{A} (A,B)=\prod_p Z_{p}(A,B)\int_0^1\mathcal{A}_{p,\theta}(A,B)~d\theta,
  \end{eqnarray*}
  where
$z_p(x):=(1-p^{-x})^{-1}$, $Z_p(A,B)=\prod_{\alpha\in
A\atop\beta\in B} z_p(1+\alpha+\beta)^{-1}$  
  and
\begin{eqnarray*}
\label{eq:Aptheta} \mathcal{A}_{p,\theta}(A,B):=  
\prod_{\alpha\in A} z_{p,-\theta}(\frac 12 +\alpha)
\prod_{\beta\in B}z_{p,\theta}(\frac 12 +\beta)
\end{eqnarray*}
with $z_{p,\theta}(x):=(1-e(\theta)p^{-x})^{-1}$.

The technique we are developing in the present series of papers is to approach our  moment problem (\ref{eqn:moment}) 
through the moments $I^\psi_{A,B}(T;X)$ of long Dirichlet polynomials 
for various ranges of $X$. The recipe of [CFKRS] also leads to  a conjectural  formula for $I^\psi_{A,B}(T;X)$. 
To explain this we begin with Perron's formula  
$$ D_{A}(s;X)=\frac{1}{2\pi i}
\int_{w} \frac{X^{w}}{w}D_{A_w}(s)~dw
$$
where we use the convenient notation 
$$A_w=\{\alpha+w:\alpha\in A\}.$$
Thus, we have
$$I^\psi_{A,B}(T;X)= \frac{1}{(2\pi i)^2}
\iint_{z,w}\frac{X^{z+w}}{zw} I^\psi_{A_{w},B_{z}}(T)~dw~dz.
$$
 We insert the conjecture above from the recipe and expect that
\begin{eqnarray*}
 I^\psi_{A,B}(T;X) &=&  T\int_0^\infty \psi(t)\frac{1}{(2\pi i)^2}
\iint_{z,w}\frac{X^{z+w}}{zw}\sum_{U\subset A, V\subset B\atop |U|=|V|}
 \left(\frac{tT}{2\pi}\right)^{-\sum_{{\alpha}\in U\atop
{\beta}\in V}({\alpha}+w+{\beta}+z) }\\
&&\qquad \qquad \times \mathcal B(A_w-U_w+V^-_z,B_z-V_z+U^-_w)~dw~dz~dt +o(T).
\end{eqnarray*}
We have done a little simplification in this expression: instead of writing $U\subset A_w$ we have written
$U\subset A$ and changed the exponent of $(tT/2\pi)$ accordingly.

Notice that there is a factor $(X/T^{|U|})^{w+z}$ here. As mentioned above we refer to  $|U|$ as the number of ``swaps" in the recipe, and now we see
more clearly the role it plays; in the terms above for which  $X<T^{|U|}$   we move the path of 
integration in $w$ or $z$ to $+\infty$ so that the factor  $(X/T^{|U|})^{w+z}\to 0$ and the contribution of such a term is 0.
Thus, the size of $X$ determines how many ``swaps'' we must keep track of. 

Our principal aim in this series of papers is to evaluate $I^\psi_{A,B}(T;X) $ directly using a conjecture for the correlations of $\tau_A(n)$ and then to 
compare with the above formula coming from the recipe of [CFKRS].  In [CK1] and [CK3] we considered the situation of 0 swaps which leads to the 
usual ``diagonal'' terms and 1 swap
which corresponds to the usual ``shifted divisor'' problem. 
 In [CK2] we considered a special case of 2 swaps. Now 
 we look at the general case of two swaps. This means that we are interested in the terms 
for which $X>T^2$ and for which 
$|U|=|V|=2$.


It is helpful to review the result of [CK3] before proceeding. The mathematical content of that paper is basically 
a conjecture and a theorem. 
First of all let $\epsilon>0$ be a small fixed number for this discussion and let $|\alpha|,|\beta|<\epsilon$ 
for all $\alpha\in A$ and $\beta\in B$.
The conjecture is about the analytic continuation of
$$\mathcal S_{A,B}(s,h):=\sum_{m=1}^\infty \frac{\tau_A(m)\tau_B(m+h)}{m^s}$$
and the sum of the residues near 1 of this:
$$\mathcal R_{A,B}(y;h):=\sum \operatornamewithlimits{Res}_{|s-1|<\epsilon}
S_{A,B}(s,h) y^{s-1}$$
where  we intend this notation to mean that $\mathcal R_{A,B}(y;h)$ is the sum of the residues of $S_{A,B}(s,h) y^{s-1}$ over all of the poles in $|s-1|<\epsilon$. 
  Let
$$ D_A(s,e(\frac 1q)):=\sum_{m=1}^\infty\frac{\tau_A(m)e(\frac m q) }{m^s}$$
and 
$$\mathcal R_A(y,q)=\sum \operatornamewithlimits{Res}_{|s-1|<\epsilon}
D_A(s,e(\frac 1q))y^{s-1} 
$$
be the sum of the residues near $s=1$, i.e.  including poles at $s=1-\alpha$ for $\alpha\in A$.
Let
$$\mathcal R_{A,B}^*(y;h):= \sum_{q=1}^\infty  r_q(h)  \mathcal R_A(y,q)\mathcal R_B(y,q)
$$
where $r_q(h)$ is the Ramanujan sum.
\begin{conjecture}
We conjecture for each fixed $h>0$ that $S_{A,B}(s,h)$ has a meromorphic continuation to $\Re s>\frac 12+\epsilon $ with all poles 
only in the region $|s-1|<\epsilon$ and that  
$$ \mathcal R_{A,B}(y;h)=\mathcal R_{A,B}^*(y;h).$$
\end{conjecture}
The above is essentially the obvious pole structure that  one would conjecture by using the $\delta$-method for example.




Now we briefly describe the calculation of [CKIII]. 
We evaluate 
$$\sum_{m,n\le X\atop m\ne n} \frac{\tau_A(m)\tau_B(n)}{\sqrt{mn}}\hat\psi\left(\frac {T}{2\pi}\log \frac mn\right)$$
as
$$2\sum_{h>0}\int_T^X \langle \tau_A(m)\tau_B(m+h)\rangle_{m\sim u} \hat\psi\left(
\frac{Th}{2\pi u}\right)\frac {du}{u}$$
which we evaluate  
by differentiating
Perron's formula with respect to $u$ and
then moving the $s$-contour to the left
to give
$$2\sum_{h>0}\int_T^X\mathcal R_{A,B}(u,h)\hat\psi\left(
\frac{Th}{2\pi u}\right)\frac {du}{u}$$
We make the change of variable $v= \frac{Th}{2\pi u}$
and rewrite this as 
$$2\int_0^\infty \hat\psi(v)\sum_{h\le \frac{2 \pi Xv}{T}} \mathcal R_{A,B}\left(\frac{Th}{2\pi v},h\right)\frac{dv}{v}$$
At this point we replace $\mathcal R$ by $\mathcal R^*$ and have 
$$2\int_0^\infty \hat\psi(v)\sum_{h\le \frac{2 \pi Xv}{T}} \sum_{q=1}^\infty 
r_q(h)\mathcal R_{A}\left(\frac{Th}{2\pi v},q\right)
\mathcal R_{B}\left(\frac{Th}{2\pi v},q\right)\frac{dv}{v}.$$
Now
$$r_q(h)=\sum_{d\mid h\atop d\mid q} d\mu(q/d)$$
so, replacing $h$ by $hd$ and $q$ by $qd$ the above is 
$$2\int_0^\infty \hat\psi(v)\sum_{q=1}^\infty \mu(q)\sum_{hd\le \frac{2 \pi Xv}{T}} 
d~\mathcal R_{A}\left(\frac{Thd}{2\pi v},qd\right)
\mathcal R_{B}\left(\frac{Thd}{2\pi v},qd\right)\frac{dv}{v}.$$
Now we express this using Cauchy's theorem as 
\begin{eqnarray*}&&
2\int_0^\infty \hat\psi(v)\sum_{q=1}^\infty  
\frac{\mu(q)}{(2\pi i)^3} \iiint_{\Re s=2\atop{|w-1|<\epsilon\atop |z-1|<\epsilon}}X^s 
\sum_{h,d=1}^\infty\left(\frac{Thd}{2\pi v}\right)^{z+w-s-2}
d\\
&&\qquad \qquad \times ~\mathcal D_{A}\left(w,e(-\frac{1}{qd})\right)
\mathcal D_{B}\left(z,e(-\frac{1}{qd})\right)\frac{dv}{v}~dz ~dw ~\frac{ds}{s}.
\end{eqnarray*}
Now we replace the sum over $h$ by $\zeta(2+s-w-z)$ and the integral over $v$ by
$$\frac{\chi(w+z-s-1)}{2}\int_0^\infty \psi(t) t^{z+w-s-2} ~dt.$$
This leads to 
\begin{eqnarray*}\int_0^\infty \psi(t)\sum_{q=1}^\infty  
\frac{\mu(q)}{(2\pi i)^3} \iiint_{\Re s=2\atop{|w-1|<\epsilon\atop |z-1|<\epsilon}}X^s 
\zeta(w+z-s-1)
\sum_{d=1}^\infty d\left(\frac{Tdt}{2\pi }\right)^{z+w-s-2} \\ \times \mathcal D_{A}\left(w,e(-\frac{1}{qd})\right)
\mathcal D_{B}\left(z,e(-\frac{1}{qd})\right)~dt ~dz ~dw ~ds.
\end{eqnarray*}
Upon comparison with the recipe we have  the identity
 \begin{theorem}
\begin{eqnarray*}&&
\operatornamewithlimits{Res}_{w=1-\alpha\atop z=1-\beta}\sum_{q=1}^\infty  
 \mu(q) 
\sum_{d=1}^\infty d^{z+w-1} \zeta(w+z-1)
\mathcal D_{A}\left(w,e(-\frac{1}{qd})\right)
\mathcal D_{B}\left(z,e(-\frac{1}{qd})\right)  \\&&
\qquad = 
  \mathcal B(A'\cup\{-\beta\},B'\cup\{-\alpha\}) .
\end{eqnarray*}
\end{theorem}

Theorem 1 follows from the
identity stated at the end of Section 3 of
[CK3] and the fact that the singular part
of $\mathcal{D}_A(s,e(\frac{1}{q}))$ is
the same as $q^{-s}\prod_{\alpha\in
A}\zeta(s+\alpha)G_A(s,q)$, as proved in
[CG].  

We call this theorem the ``analytic version of the general shifted divisor sum.'' 
In this paper we prove an identity that is an analogue of Theorem 1 but for a convolution of two 
shifted divisor sums. This is a step forward in this process of understanding moments. The key theorem is a convolution identity
\begin{theorem}
\begin{eqnarray*}&&
\operatornamewithlimits{Res}_{{w_1=1-\alpha_1\atop z_1=1-\beta_1}\atop
{w_2=1-\alpha_2\atop z_2=1-\beta_2}}
\zeta(w_1+z_1-1)\zeta(w_2+z_2-1)
\sum_{(M,N)=1\atop {d_1,d_2\atop q_1,q_2}}\frac{ \mu(q_1)\mu(q_2)d_1^{z_1+w_1-1}
d_2^{z_2+w_2-1}}{M^{z_1+w_2-1}
N^{w_1+z_2-1}} 
  \\
&&\qquad
\mathcal D_{A_1}\left(w_1,e(-\frac{N}{q_1d_1})\right)\mathcal D_{A_2}\left(w_2,e(-\frac{M}{q_2d_2})\right)
\mathcal D_{B_1}\left(z_1,e(-\frac{M}{q_1d_1})\right) \mathcal D_{B_2}\left(z_2,e(-\frac{N}{q_2d_2})\right)  \\&&
\qquad = 
  \mathcal B(A''\cup\{-\beta_1,-\beta_2\},B''\cup\{-\alpha_1,-\alpha_2\}) .
\end{eqnarray*}
\end{theorem}
 
 Theorem 2 follows from Section 11
because if $(a,q)=1$ then the singular
part of $\mathcal{D}_A(s,e(\frac{a}{q}))$
is identical to that of
$q^{-s}\prod_{\alpha\in
A}\zeta(s+\alpha)G_A(s,q)$.

A particularly interesting feature of this theorem is the appearance of the sum over $M$ and $N$.
It is these parameters which prompt us to liken this calculation to a circle method calculation. Basically the $M$ and $N$ make their appearance because of a splitting of the equation 
$m_1m_2-n_1n_2=h$ into a pair of equations where $m_1/n_1\approx M/N\approx n_2/m_2$
which gives $Mm_1-Nn_1=h_1, Nm_2-Mn_2=h_2$. This is the fundamental new idea of the paper.

In the next section of this paper we present the basic set up, which involves a convolution of two shifted divisor sums.  In sections 4, 5 and 6 we deal with the semi-diagonal case where one of the shifted divisor sums is degenerate.  In sections 7, 8 and 9 we motivate heuristically the identity of Theorem 2.
This identity is sufficiently
complicated that we find it convenient to recast it as an  
equality of  certain power series. 
Sections 10 and 11 are devoted to the 
rigorous proof of this identity.

\section{Type II convolution sums}
To proceed we  approach the moment $I_{A,B}^\psi(T;X)$ 
  through arithmetic means.
 To do this, we consider a convolution of shifted correlation sums.

We first make use of the fact that if $A=A_1\cup A_2$ and $B=B_1\cup B_2$ then $\tau_A$ and $\tau_B$ are convolutions: 
$\tau_A=\tau_{A_1}*\tau_{A_2}$ and $\tau_B=\tau_{B_1}*\tau_{B_2}$. We are thus interested in  
$$\mathcal O_{II}=\sum_{m_1m_2,n_1n_2 \le X\atop
0<|m_1m_2-n_1n_2|<m_1m_2/\tau} \frac{\tau_{A_1}(m_1)\tau_{A_2}(m_2) \tau_{B_1}(n_1)\tau_{B_2}(n_2)}
 {m_1m_2}  \hat \psi \big(\frac{T}{2\pi}\log((n_1n_2)/(m_1m_2))\big).$$
Now we embark on a discrete analog of the circle method which basically consists of approximating a ratio, say $m_1/n_1$ by 
a rational number with a small denominator, say $M/N$, and then sum all of the terms with $m_1/n_1$ close to $M/N$.

 To this end  
we introduce a parameter $Q$ and subdivide the interval $[0,1]$ into Farey intervals associated with  the fractions $M/N$ with 
$1\le M\le N\le Q$ and $(M,N)=1$ from the Farey sequence $\mathcal F_Q$; see \cite{CKII} for details.
 We
  define 
$$h_1 := m_1 N - n_1 M$$
and  $$h_2: = m_2 M -n_2 N.$$
We have
$$ m_1m_2 MN -n_1n_2 MN = h_1m_2M+h_2m_1 N -h_1 h_2$$
so that 
$$
\frac{m_1 m_2 -n_1 n_2}{m_1m_2}=\frac{h_1}{m_1N}+\frac{h_2}{m_2 M} -\frac{h_1h_2}{m_1m_2MN}
$$
and 
$$\log\frac{n_1n_2}{m_1m_2}=\frac{h_1}{m_1N}+\frac{h_2}{m_2 M} +O\big(\frac{h_1h_2}{m_1m_2MN}\big).
$$
The error term is negligible so we 
have now  arranged the sum as 
\begin{eqnarray} \label{eqn:start}
\sum_{M\le N\le Q\atop (M,N)=1}
\sum_{h_1,h_2 }
\sum_{{m_1m_2\le X }\atop {(*_1), (*_2) }}
\frac{\tau_{A_1}(m_1)\tau_{A_2}(m_2) \tau_{B_1}(n_1)\tau_{B_2}(n_2)}{m_1m_2}
 \hat \psi  \left(\frac{Th_1}{2\pi m_1N}+\frac{Th_2}{2\pi m_2 M} \right)
\end{eqnarray}
 where 
$$ (*_1): m_1N-n_1M=h_1 \qquad \qquad \mbox{ and } \qquad \qquad  (*_2): m_2M-n_2N=h_2
$$
Note that for a given $m_1, n_1$ and $h_1$ the condition $(*_1)$ implies that $m_1/n_1 \in \mathcal M_{M,N}$
so we don't need to write that condition. 

\section{The case of $h_2=0$} 
We remark first of all that the terms with $h_1=h_2=0$ are precisely the diagonal terms. Now we consider what happens if $h_2=0$ and $h_1\ne0$.
We call this a ``semi-diagonal'' term after [BK].

If $h_2=0$ then $m_2M=n_2N$. Since $(M,N)=1$ it follows that
$m_2=N\ell$ and $n_2=M\ell$ for some $\ell$.
Thus we have
$$\sum_{M\le N\atop (M,N)=1}\phi\left( \frac M Q\right)  \phi\left( \frac N Q\right)
\sum_{ h_1}
\sum_{{m_1, n_1 ,\ell }\atop {(*_1) \atop n_1\ge |h_1|Q }}
\frac{\tau_{A_1}(m_1)\tau_{A_2}(N\ell) \tau_{B_1}(n_1)\tau_{B_2}(M\ell)}{m_1N\ell}
 \hat \psi  \left(\frac{Th_1}{2\pi m_1N} \right)
$$
where
$$(*_1): m_1N-n_1M=h_1 .$$
In general, with $*:mN-nM=h $, we expect by the delta-method that 
\begin{eqnarray}&& \label{eqn:delta}
\langle \tau_{A}(m)\tau_{B}(n)\rangle^{(*)}_{m= u}\sim
\sum_{\alpha\in A\atop  \beta\in B}u^{-\alpha- \beta}M^{-1+\beta}N^{-\beta}
Z(A'_{-\alpha})Z(B'_{-\beta})
\\&& \quad \times
 \sum_{d\mid h} \frac 1 {d^{1-\alpha- \beta}} \sum_{q} \frac{\mu(q)(qd,M)^{1-\beta}
(qd,N)^{1- \alpha}}{q^{2-\alpha-\beta}} 
  G_{A}\big(1-\alpha,\frac{qd}{(qd,N)}\big)  
  G_{B}\big(1-\beta,\frac{qd}{(qd,M)}) ,\nonumber
\end{eqnarray}
where $G$ is a multiplicative function for which
$$G_A(1- \alpha,p^r)=\prod_{\hat \alpha\in A'}\left(1-\frac 1{p^{1+\hat \alpha-\alpha}}\right)
\sum_{j=0}^\infty \frac{\tau_{A'}(p^{j+r})}{p^{j(1- \alpha)}}$$
with $A'=A-\{ \alpha\}$  and where
$$Z(A)=\prod_{a\in A}\zeta(1+a).$$
 
\section{A diversion}
The ensuing calculations are about to become (more) complicated largely due to arithmetic factors.
We pause in the calculation to show what the calculations look like without
the arithmetic factors. That should help the reader when we complete this calculation in the next section.
 Basically we ignore the
terms with $q\geq2$ and we replace
$G_A(1-\alpha,d)$ by $\tau_{A'}(d)$.

Altogether we now have 
\begin{eqnarray*} &&
\sum_{ \alpha \in A_1\atop 
 \beta\in B_1} Z((A_1')_{-\alpha})Z((B_1')_{- \beta})   
\sum_{M\le N\atop (M,N)=1}\phi\left( \frac M Q\right)  \phi\left( \frac N Q\right)
\sum_{\ell, h_1} \frac{\tau_{A_2}(\ell)\tau_{B_2}(\ell)}{\ell}
\\&& \qquad \times \int_{u\le \frac{X}{N\ell}} \sum_{d\mid h_1}
 \frac{(d,N)^{1-{\alpha}}(d,M)^{1-{\beta}}
\tau_{A_1'}\big(\frac{d}{(d,N)}\big)\tau_{B_1'}\big(\frac{d}{(d,M)}\big)\tau_{A_2}(N)\tau_{B_2}(M)}{
M^{1- \beta}N^{1+ \beta}u^{\alpha+\beta}d^{1-\alpha- \beta}}
 \hat \psi  \left(\frac{Th_1}{2\pi uN} \right)\frac{du}{u}.
\end{eqnarray*}
We make the substitution 
$$v=\frac{Th_1}{2\pi uN}.$$
The above is 
 \begin{eqnarray*} &&
\sum_{ \alpha \in A_1\atop 
 \beta\in B_1}Z((A_1')_{-\alpha})Z((B_1')_{- \beta})  
\sum_{M\le N\atop (M,N)=1}\phi\left( \frac M Q\right)  \phi\left( \frac N Q\right)
\sum_{\ell, h_1} \frac{\tau_{A_2}(\ell)\tau_{B_2}(\ell)}{\ell}
\\&& \qquad \times \int_{v\ge \frac{Th_1\ell}{2\pi X}} \sum_{d\mid h_1}
 \frac{(d,N)^{1-{\alpha}}(d,M)^{1-{\beta}}
\tau_{A_1'}\big(\frac{d}{(d,N)}\big)\tau_{B_1'}\big(\frac{d}{(d,M)}\big)\tau_{A_2}(N)\tau_{B_2}(M)}{
M^{1- \beta}N^{1+ \beta}\left(\frac{Th_1}{2\pi vN}\right) ^{\alpha+\beta}d^{1-\alpha- \beta}}
 \hat \psi (v)\frac{dv}{v}.
\end{eqnarray*}
Now we switch the sums around; replacing $h_1$ by $h_1d$ and bringing the sum over $h_1$ and $\ell$ to the inside, we have 
 \begin{eqnarray*} &&
\sum_{ \alpha \in A_1\atop 
 \beta\in B_1}\left(\frac{T}{2\pi}\right)^{- \alpha- \beta}Z((A_1')_{-\alpha})Z((B_1')_{- \beta})  
\sum_{M\le N\atop (M,N)=1}\frac{\phi\left( \frac M Q\right)  \phi\left( \frac N Q\right)\tau_{A_2}(N)\tau_{B_2}(M)
}{M^{1- \beta}N^{1- \alpha}}
\\&& \qquad \times 
\int_v  
 \frac{ \hat \psi (v)}{v^{1-\alpha- \beta}}
\sum_{d\ell h_1 \le \frac{2\pi Xv}{T}} \frac{(d,N)^{1-{\alpha}}(d,M)^{1-{\beta}}
\tau_{A_1'}\big(\frac{d}{(d,N)}\big)\tau_{B_1'}\big(\frac{d}{(d,M)}\big)\tau_{A_2}(\ell)\tau_{B_2}(\ell)}{
 d \ell h_1^{ \alpha+ \beta}}
 ~dv.
\end{eqnarray*}
Using Perron's formula we write this as 
\begin{eqnarray*} &&
\sum_{ \alpha \in A_1\atop 
 \beta\in B_1}\left(\frac{T}{2\pi}\right)^{- \alpha- \beta}Z((A_1')_{-\alpha})Z((B_1')_{- \beta})  
\sum_{M\le N\atop (M,N)=1}\frac{\phi\left( \frac M Q\right)  \phi\left( \frac N Q\right)\tau_{A_2}(N)\tau_{B_2}(M)
}{M^{1- \beta}N^{1- \alpha}}
\\&& \qquad \times 
\int_v  
 \frac{ \hat \psi (v)}{v^{1-\alpha- \beta}}
\frac{1}{2\pi i}\int_{(2)}\sum_{d,\ell, h_1 } \frac{(d,N)^{1-{\alpha}}(d,M)^{1-{\beta}}
\tau_{A_1'}\big(\frac{d}{(d,N)}\big)\tau_{B_1'}\big(\frac{d}{(d,M)}\big)\tau_{A_2}(\ell)\tau_{B_2}(\ell)}{
 d^{s+1}\ell^{s+1} h_1^{s+ \alpha+ \beta}}
\\&&\qquad \qquad \qquad 
\left( \frac{2\pi Xv}{T}\right)^s\frac{ds}{s}
 ~dv.
\end{eqnarray*}
The sum over $\ell$ and $h_1$ here is essentially
$$\zeta(s+\alpha+ \beta) Z((A_2)_s,B_2).$$
The sum over $d$, $M$ and $N$ we 
 evaluate to a first approximation by looking at the polar parts of 
\begin{eqnarray*} 
\sum_{d, M,N\atop (M,N)=1}
\frac{(d,M)^{1-\beta}(d,N)^{1-\alpha}\tau_{A_1'}(\frac{d}{(d,N)}) 
\tau_{B_1'}(\frac{d}{(d,M)})\tau_{A_2}(N)
 \tau_{B_2}(M)}
{
 d^{1+s}M^{1-\beta} 
N^{1-\alpha} };
 \end{eqnarray*}
these  are  calculated with the help of the following table:  
\begin{eqnarray*}
\begin{array}{|l|l|l|l|l|l|}
\hline
d&M&N& \mbox{Euler term}& Z-\mbox{factor}\\
\hline
p&1&1&\tau_{A_1'}(p)\tau_{B_1'}(p)/p^{1+s}&Z((A_1')_{s}, B_1')\\
\hline
1&p&1& \tau_{B_2}(p)/p^{1-\beta } &Z(B_2,\{- \beta\})\\
\hline
1&1&p&\tau_{A_2}(p)/p^{1-\alpha } &Z(A_2,\{- \alpha\})\\
\hline
p&1&p& \tau_{B_1'}(p)\tau_{A_2}(p)/p^{1+s}&Z(B_1', (A_2)_{s})\\
\hline
p&p&1& \tau_{A_1'}(p)\tau_{B_2}(p)/p^{1+s} &Z((A_1')_{s}, B_2)\\
\hline
\end{array}
\end{eqnarray*}
We take the product of all of these $Z$ factors.
Now the $v$-integral is 
$$\int_v  
 \frac{ \hat \psi (v)}{v^{1-s-\alpha- \beta}}~dv=(1/2) \chi(1-s-\alpha-\beta) \int_0^\infty \psi(t) 
t^{-s-\alpha-\beta}~dt.
$$
 Note that
$$   \chi(1-s-\alpha-\beta)\zeta(s+\alpha+ \beta)=
 \zeta(1-s-\alpha-\beta).$$

If we include the factors $Z((A_2)_s,B_2)$, $Z((A_1')_{-\alpha})Z((B_1')_{- \beta})  $, 
and $Z(\{-s-\alpha\}, \{- \beta\})$ 
then the product of all of these $Z$-factors is 
\begin{eqnarray*} 
Z\bigg((A_1' \cup A_2)_{s} \cup \{- \beta\}  
, B_1' \cup B_2 \cup \{-s-\alpha\}  \bigg)=Z\bigg((A')_{s} \cup \{- \beta\}  
, B' \cup \{-s-\alpha\}  \bigg).
\end{eqnarray*}

 Thus, altogether we have
\begin{eqnarray*} &&
\sum_{ \alpha \in A_1\atop 
 \beta\in B_1}\left(\frac{T}{2\pi}\right)^{- \alpha- \beta}
\int_t  
 \frac{ \psi(t)}{t^{\alpha+ \beta}}
\frac{1}{2\pi i}\int_{(2)}Z\bigg((A')_{s} \cup \{- \beta\}  
, B' \cup \{-s-\alpha\}  \bigg) \left( \frac{2\pi X}{tT}\right)^s\frac{ds}{s} ~dt.
\end{eqnarray*}
 Compare this with equation (4) of [CK3] which gives the ``one-swap'' terms from the recipe:
\begin{eqnarray*} &&
\int_0^\infty \psi(t) \sum_{\alpha\in A\atop  \beta \in B}
\left(\frac {Tt}{2\pi}\right)^{- \alpha - \beta} Z((A')_{-\alpha})Z((B')_{- \beta})  \\
&& \qquad \qquad \times \frac{1}{2\pi i}
\int_{\Re s=4}
\frac{\left(\frac {2\pi X}{Tt}\right)^s}{s} \mathcal A(A'\cup\{-\beta-s\},B'_s\cup\{-\alpha\})
 Z(A'_{s},B')  \zeta(1-\alpha-\beta -s)~ds. \nonumber
\end{eqnarray*}
The only differences are that so far we have ignored the arithmetic factors
 and that in the expression we just derived we have the restrictions $\alpha \in A_1$ and 
$ \beta \in B_1$. But as $A_1$ and $B_1$ vary through subsets of $A$ and $B$ every possible $\alpha$ and 
$ \beta$ will appear. Also, we have the terms where $M>N$ and those with $h_1=0$.  
 
\section{The same calculation with the arithmetic factors}
We replace $m_1$ by   $u_1$; taking into account the arithmetic considerations and also using
  $u_1\ell  N=m_1m_2\le X$, we have that our sum is 
  $\sum_{\alpha,\beta} Z((A'_1)_{-\alpha})Z((B'_1)_{-\beta}) \times $
\begin{eqnarray*} &&
\sum_{M\le N\atop (M,N)=1} \phi\left( \frac M Q\right)  \phi\left( \frac N Q\right) 
\sum_{h_1}
\sum_{ \ell  }\frac{\tau_{A_2}(N\ell )\tau_{B_2}(M\ell )}{M^{1-\beta}N^{1+ \beta}\ell } \int_{u_1\ell \le \frac{X}{ N} }  \sum_{d\mid h_1 } 
\sum_{q=1}^\infty \frac{\mu(q)(qd,M)^{1-\beta}
(qd,N)^{1- \alpha}}{d^{1- \alpha-\beta}q^{2- \alpha- \beta}}\\
&&\qquad \times G_{A_1}\big(1-\alpha,\frac{qd}{(qd,N)}\big)  
  G_{B_1}\big(1-\beta,\frac{qd}{(qd,M)})    
 \hat \psi  \left(\frac{Th_1}{2\pi u_1N} \right)~\frac{du_1}{u_1^{1+\alpha+\beta}}.
\end{eqnarray*}

The term with $h_1=0$ just leads to diagonal terms which are easy to deal with. 
Now we group the non-zero terms  $h_1$ and $-h_1$ together and
 use $ \psi(-v)=\overline { \psi(v)}$.  We replace $h_1$ by $h_1d$.
We make the substitution $v_1=\frac{Th_1d}{2\pi u_1N}$ in the integral and switch the integral 
over $v_1$ with the sum over $h_1$, $d$  and $\ell$. 
Then (with $h_1>0$) we have that  
$$\frac{\ell NTh_1d}{2\pi v_1N}=u_1\ell N\le X
$$
implies that $$\ell h_1 d\le \frac{2\pi Xv_1}{T}.$$
Thus we have
\begin{eqnarray*}
&&\sum_{\alpha \in A_1\atop
 \beta\in B_1} 
\left(\frac{T}{2\pi}\right)^{-\alpha- \beta} Z((A_1')_{- \alpha})Z((B_1')_{ -\beta})
\sum_{M\le N\atop (M,N)=1} \frac{\phi\left( \frac M Q\right)  \phi\left( \frac N Q\right) }{M^{1-\beta}
N^{1-\alpha}}
 \int_{0}^\infty   (2\Re {\hat \psi }(v_1))
\\&&\qquad \times
 \sum_{h_1\ell d\le \frac{2\pi Xv_1}{T}}\frac{\tau_{A_2}(N\ell)\tau_{B_2}(M\ell)}
{h_1^{\alpha+\beta}  \ell d}
\sum_{q\ge1} \frac{\mu(q)(qd,M)^{1-\beta}(qd,N)^{1-\alpha}}{q^{2-\alpha- \beta}} 
 \\
&&\qquad \qquad \times 
   G_{{A_1}}\left(1- \alpha,\frac {qd}{(qd,N)} \right)  
G_{{B_1}}\left(1-\beta,\frac {qd}{(qd,M)} \right)  
  ~\frac{dv_1}{v_1^{1-\alpha-\beta}}.
\end{eqnarray*}

Now we use Perron's formula to evaluate the sum over $h_1\ell d$. This gives
\begin{eqnarray*}
&&\sum_{\alpha \in A_1\atop
 \beta\in B_1}\left(\frac{T}{2\pi}\right)^{- \alpha - \beta} Z((A_1')_{- \alpha})Z((B_1')_{ -\beta})
\sum_{M\le N\atop (M,N)=1} \frac{\phi\left( \frac M Q\right)  \phi\left( \frac N Q\right) }{M^{1-\beta}N^{1- \alpha}}
\frac{1}{2\pi i}\int_{(2)}
\\&&\qquad \times \int_{0}^\infty   (2\Re {\hat \psi }(v_1)) \left(\frac{2\pi Xv_1}{T}\right)^s
 \sum_{h_1,\ell, d} \frac{\tau_{A_2}(N\ell)\tau_{B_2}(M\ell)}
{h_1^{s+\alpha+\beta}  \ell^{1+s} d^{1+s } }
\sum_{q\ge1} \frac{\mu(q)(qd,N)^{1-\alpha}(qd,M)^{1- \beta}}{q^{2-\alpha- \beta}}    \\
&&\qquad \qquad \times 
   G_{{A_1}}\left(1- \alpha,\frac {qd}{(qd,N)} \right)  
G_{{B_1}}\left(1-\beta,\frac {qd}{(qd,M)} \right)  
    ~\frac{dv_1}{v_1^{1- \alpha- \beta}}\frac{ds}{s}.
\end{eqnarray*}
The sum over $h_1$ is $\zeta(s+\alpha+\beta)$. The integral over $v_1$  is 
\begin{eqnarray*}
\int_{0 }^{\infty} v_1^{-1+s+ \alpha+ \beta}
(2\Re  \hat \psi (v_1))~dv_1
 =\chi(1-s- \alpha - \beta ) \int_0^\infty \psi(t)t^{-s-\alpha - \beta}~dt.
\end{eqnarray*}
Combining these two facts and using the functional equation for $\zeta$ we have
 \begin{eqnarray*}
&&\sum_{\alpha \in A_1\atop
 \beta\in B_1}\left(\frac{T}{2\pi}\right)^{- \alpha - \beta} Z((A_1')_{- \alpha})Z((B_1')_{ -\beta})
\int_{0 }^{\infty} t^{- \alpha- \beta}
 \psi(t)
\sum_{M\le N\atop (M,N)=1} \frac{\phi\left( \frac M Q\right)  \phi\left( \frac N Q\right) }{M^{1- \beta}N^{1-\alpha}}
\\&&  \times  \frac{1}{2\pi i}\int_{(2)}\zeta(1-s-\alpha- \beta)
 \left(\frac{2\pi X}{tT}\right)^s\sum_{\ell, d} \frac{\tau_{A_2}(N\ell)\tau_{B_2}(M\ell)}
{   \ell^{1+s} d^{1+s } }
  \\
&&\quad  \times \sum_{q\ge1} \frac{\mu(q)(qd,N)^{1-\alpha}(qd,M)^{1- \beta}}{q^{2-\alpha- \beta}}  
   G_{{A_1}}\left(1- \alpha,\frac {qd}{(qd,N)} \right)  
G_{{B_1}}\left(1-\beta,\frac {qd}{(qd,M)} \right)  
    ~ \frac{ds}{s}~dt.
\end{eqnarray*}
 
This requires studying the Dirichlet series
\begin{eqnarray*}
&& 
\sum_{  (M,N)=1} \frac{1 }{M^{1-\beta}N^{1-\alpha}}
 \sum_{\ell, d} \frac{\tau_{A_2}(N\ell)\tau_{B_2}(M\ell)}
{  \ell^{1+s} d^{1+s } }
\sum_{q\ge1} \frac{\mu(q)(qd,N)^{1-\alpha}(qd,M)^{1- \beta}}{q^{2-\alpha- \beta}}    \\
&&\qquad \qquad \times 
   G_{{A_1}}\left(1- \alpha,\frac {qd}{(qd,N)} \right)  
G_{{B_1}}\left(1-\beta,\frac {qd}{(qd,M)} \right)  
\end{eqnarray*}
 
See the appendix for the resolution of this arithmetic factor.

Regarding multiplicities, see the section on automorphisms at the end of the paper. 

Taking account of the terms with $h_1=0$ and $h_2\ne 0$ we find that 
we have now accounted for all of the one-swap terms from the semi-diagonal contributions.

\section{$h_1h_2\ne 0$}
Now we come to the crux of the paper, the terms where neither $h_1$ nor $h_2$ are 0;
we need to match these up with the two swap terms. 

In the formula (\ref{eqn:start}) above
we replace the convolution sums by their averages, i.e.
\begin{eqnarray*}
 \iint_{u_1u_2\le X  }\langle \tau_{A_1}(m_1)\tau_{B_1}(n_1)\rangle^{(*_1)}_{m_1\sim u_1}
\langle \tau_{A_2}(m_2)\tau_{B_2}(n_2)\rangle^{(*_2)}_{m_2\sim u_2} 
 \hat \psi  \left(\frac{Th_1}{2\pi u_1N}+\frac{Th_2}{2\pi u_2 M} \right)~\frac{du_1}{u_1}\frac{du_2}
 {u_2}.
\end{eqnarray*}
 
We insert the formula (\ref{eqn:delta}) for these averages. After switching the ensuing sums over $h_1,h_2$ and 
$d_1,d_2$ we have $\sum_{M,N} \phi(M/Q)\phi(N/Q)$ times 
 \begin{eqnarray*}&&
\sum_{{\alpha_1}\in A_1\atop \alpha_2\in A_2} 
\sum_{{\beta_1}\in B_1\atop \beta_2\in B_2} Z((A_1')_{- \alpha_1})Z((A_2')_{- \alpha_2})
Z((B_1')_{- \beta_1})Z((B_2')_{- \beta_2}) \sum_{q_1,d_1,h_1 \atop 
q_2,d_2,h_2} \frac{\mu(q_1)\mu(q_2)}{q_1^{2-\alpha_1- {\beta}_1}q_2^{2-\alpha_2- {\beta}_2}}
 \\
&&
\times  \frac{G_{A_1}(1- \alpha_1, \frac{q_1d_1}{(q_1d_1,N)}) 
G_{A_2}(1- \alpha_2 ,\frac{q_2d_2}{(q_2d_2,M)})
G_{B_1}(1- \beta_1, \frac{q_1d_1}{(q_1d_1,M)}) G_{B_2}(1- \beta_2 ,\frac{q_2d_2}{(q_2d_2,N)})}
{(q_1d_1,N)^{-1+\alpha_1}(q_1d_1,M)^{-1+\beta_1}(q_2d_2,M)^{-1+\alpha_2}(q_2d_2,N)^{-1+\beta_2}
d_1^{1- \alpha_1 - \beta_1}
d_2^{1- \alpha_2 - \beta_2}
}
\\
&& \qquad  \times 
 \iint_{T^2\le u_1u_2\le X}M^{-1+\beta_1-\beta_2}N^{-1+ \beta_2 -\beta_1}  u_1^{- \alpha_1 - \beta_1}
u_2^{- \alpha_2 - \beta_2}\hat \psi  \left(\frac{Th_1d_1}{2\pi u_1N}+
\frac{Th_2d_2}{2\pi u_2M}\right)\frac{du_1}{u_1}\frac{du_2}{u_2}.
\end{eqnarray*}
Let's first assume that $h_1>0$ and $h_2>0$. 
We make the changes of variable $v_1=\frac{Th_1d_1}{2\pi u_1N}$ and $v_2=\frac{Th_2d_2}{2\pi u_2M}$  and bring the sums over $h_1$ 
and $h_2$ to the inside; $u_1u_2<X$ implies that 
$$ h_1d_1h_2d_2< \frac{4\pi^2 Xv_1v_2MN}{T^2}.$$ 
Then the sums over the $q_i,h_i,d_i$  are $N^{-1+\alpha_1+\beta_2} M^{-1+\alpha_2+\beta_1}$ times
\begin{eqnarray*}&&\left(\frac{T}{2\pi}\right)^{- \alpha_1 - \alpha_2-
 \beta_1 - \beta_2}
\iint_{v_1,v_2} 
 v_1^{ \alpha_1 + \beta_1}v_2^{ \alpha_2 + \beta_2}\hat \psi (v_1+v_2)
\frac{1}{2\pi i} \int_{(2)} \sum_{q_1,d_1,h_1 \atop 
q_2,d_2,h_2}  \frac{\mu(q_1)\mu(q_2)}{q_1^{2-\alpha_1- {\beta_1}}
q_2^{2-\alpha_2- {\beta}_2}}
\\&&\qquad \times \frac{G_{A_1}(1- \alpha_1, \frac{q_1d_1}{(q_1d_1,N)}) G_{A_2}(1- \alpha_2 ,\frac{q_2d_2}{(q_2d_2,M)})
G_{B_1}(1- \beta_1, \frac{q_1d_1}{(q_1d_1,M)}) G_{B_2}(1- \beta_2 ,\frac{q_2d_2}{(q_2d_2,N)})}
{(q_1d_1,N)^{-1+\alpha_1}(q_1d_1,M)^{-1+\beta_1}(q_2d_2,M)^{-1+\alpha_2}
(q_2d_2,N)^{-1+\beta_2}d_1^{1+s}
d_2^{1+s}h_1^{\alpha_1+\beta_1+s}h_2^{\alpha_2+\beta_2+s}
 }
\\&&\qquad \qquad \qquad \times \frac{\left(\frac{4\pi^2 Xv_1v_2MN}{T^2}\right)^s}{s}
~ds
\frac{dv_1}{v_1}\frac{dv_2}{v_2}.
\end{eqnarray*}
The sums over $h_1$ and $h_2$ are $\zeta(s+\alpha_1+\beta_1)\zeta(s+\alpha_2+\beta_2)$.
The other 3 cases of the signs of $h_1$ and $h_2$ can be taken care of similarly. 
Then we use 
$$\hat \psi (v_1+v_2)=\int_0^\infty \psi(t)e(t(v_1+v_2))~dt$$
to see that 
\begin{eqnarray*}&&
 \hat \psi (v_1+v_2)+ \hat \psi (v_1-v_2)+ \hat \psi (-v_1+v_2)+ \hat \psi (-v_1-v_2)\\&&\qquad \qquad
=\int_0^\infty \psi(t)\big(e(tv_1)+e(-tv_1)\big)\big(e(tv_2)+e(-tv_2)\big)~dt.
\end{eqnarray*}
Also
$$ \int_0^\infty v_1^{s+\alpha+\gamma-1}(e(tv_1)+e(-tv_1))~dv_1 =t^{-s-\alpha-\gamma}\chi(1-s-\alpha-\gamma),$$ 
and similarly for the integral over $v_2$.
This leaves us with a    total for the sum over $M,N,q_i,h_i,d_i$  of 
\begin{eqnarray*}&&
\int_0^\infty \psi(t)  
\left(\frac{tT}{2\pi}\right)^{- \alpha_1 - \alpha_2-
 \beta_1 - \beta_2}
\frac{1}{2\pi i} \int_{(2)}\frac{\left(\frac{4\pi^2 X }{t^2T^2}\right)^s}{s}
\zeta(1-s-\alpha_1- \beta_1)\zeta(1-s-\alpha_2- \beta_2)
\\&&\qquad \sum_{(M,N)=1\atop M\le N}\frac{\phi(M/Q)\phi(N/Q)}{M^{1-s-\alpha_2-\beta_1}N^{1-s-\alpha_1-\beta_2}}
 \sum_{q_1,d_1 \atop 
q_2,d_2} \frac{\mu(q_1)\mu(q_2)}{q_1^{2-\alpha_1- {\beta}_1}q_2^{2-\alpha_2- {\beta}_2}}
\frac{G_{A_1}(1- \alpha_1, \frac{q_1d_1}{(q_1d_1,N)})}
{(q_1d_1,N)^{-1+\alpha_1}}
\\&&\qquad  \qquad \times\frac{  G_{A_2}(1- \alpha_2 ,\frac{q_2d_2}{(q_2d_2,M)})
G_{B_1}(1- \beta_1, \frac{q_1d_1}{(q_1d_1,M)}) G_{B_2}(1- \beta_2 ,\frac{q_2d_2}{(q_2d_2,N)})}
{
 (q_1d_1,M)^{-1+\beta_1}(q_2d_2,M)^{-1+\alpha_2}(q_2d_2,N)^{-1+\beta_2}d_1^{1+s}
d_2^{1+s} 
}~ds ~dt.
\end{eqnarray*}
Recall that
\begin{eqnarray*}
G_A(1- \alpha , p)&=&  \tau_{A'}(p)+O(1/p);
\end{eqnarray*}
we use this to calculate  the polar part of 
\begin{eqnarray*} 
\sum_{d_1,d_2\atop {(M,N)=1\atop M\leq N}}
\frac{\tau_{A_1'}(\frac{d_1}{(d_1,N)}) 
\tau_{B_1'}(\frac{d_1}{(d_1,M)})\tau_{A_2'}(\frac{d_2}{(d_2,M)})
 \tau_{B_2'}(\frac{d_2}{(d_2,N)})}
{(d_1,N)^{-1+\alpha_1}(d_1,M)^{-1+\beta_1}(d_2,M)^{-1+\alpha_2}(d_2,N)^{-1+\beta_2}
 d_1^{1+s}d_2^{1+s}M^{1-s-\alpha_2 - \beta_1} 
N^{1-s-\alpha_1 - \beta_2} }.
 \end{eqnarray*}

We do this by calculating the significant parts of the Euler product. The following table is helpful;
we let $A_1'=A_1-\{\alpha_1\}$, $A_2'=A_2-\{\alpha_2\}$, $B_1'=B_1-\{ \beta_1\}$, and $B_2'=B_2-\{ \beta_2\}$. 
\begin{eqnarray*}
\begin{array}{|l|l|l|l|l|l|}
\hline
d_1&d_2&N&M& \mbox{Euler term}& Z-\mbox{factor}\\
\hline
p&1&1&1&\tau_{A_1'}(p)\tau_{B_1'}(p)/p^{1+s}&Z((A_1')_s, B_1')\\
\hline
1&p&1&1&\tau_{A_2'}(p)\tau_{B_2'}(p)/p^{1+s}&Z((A_2')_s, B_2')\\
\hline
1&1&p&1&p^{ \alpha_1+ \beta_2}/p^{1-s}&Z(\{- \alpha_1-s\}, \{- \beta_2\})\\
\hline
1&1&1&p&p^{ \alpha_2+ \beta_1}/p^{1-s}&Z(\{- \alpha_2-s\}, \{- \beta_1\})\\
\hline
p&1&p&1& \tau_{B_1'}(p)/p^{1- \beta_2}&Z(B_1', \{- \beta_2\})\\
\hline
p&1&1&p& \tau_{A_1'}(p)/p^{1- \alpha_2 }&Z(A_1', \{- \alpha_2\})\\
\hline
1&p&p&1& \tau_{A_2'}(p)/p^{1- \alpha_1}&Z(A_2', \{- \alpha_1\})\\
\hline
1&p&1&p& \tau_{B_2'}(p)/p^{1- \beta_1}&Z(B_2', \{- \beta_1\})\\
\hline
p&p&p&1& \tau_{A_2'}(p)\tau_{B_1'}(p)/p^{1+s} &Z((A_2')_s, B_1')\\
\hline
p&p&1&p& \tau_{A_1'}(p)\tau_{B_2'}(p)/p^{1+s}&Z((A_1')_s, B_2') \\
\hline
\end{array}
\end{eqnarray*}
If we include the factors $Z((A_1')_s,\{-s-\alpha_1\})Z(B_1',\{-\beta_1\})$, 
$Z((A_2')_s,\{-s-\alpha_2\})Z(B_2',\{-\beta_2\})$,
$Z(\{-s-\alpha_1\}, \{- \beta_1\})$ and $Z(\{-s-\alpha_2\}, \{- \beta_2\})$
then the product of all of these $Z$-factors is 
\begin{eqnarray*}&& 
Z\big((A_1' \cup A_2')_s \cup \{- \beta_1\} \cup \{- \beta_2\}
, B_1' \cup B_2'\cup \{-s- \alpha_1\}\cup \{-s- \alpha_2\}\big)\\
&&\qquad
=Z((A-S)_s+T^-,B-T+(S_s)^-)
\end{eqnarray*}
where $S=\{\alpha_1, \alpha_2\}$ and $T=\{ \beta_1,\beta_2\}$.
 
The predicted two-swap terms from the recipe are
\begin{eqnarray*}&&
\sum_{S\subset A,T\subset B
\atop |S|=|T|=2}\int_0^\infty \psi(t)\left(\frac{tT}{2\pi}\right)^{-\sum_{ \alpha\in S
\atop  \beta\in T}( \alpha +\beta)} 
\frac{1}{2\pi i} \int_{(2)}\frac{\left(\frac{4\pi^2 X }{t^2T^2}\right)^s}{s}
\mathcal AZ((A-S)_s+T^-,B-T+S_s^-)~ds ~dt
\end{eqnarray*}
which matches the above except that $S$ and $T$ are allowed to range over all two-element subsets 
of $A$ and $B$ in the recipe version whereas in the correlation version we first split $A=A_1\cup A_2$
and $B=B_1\cup B_2$ and then take one element from $A_1$ and one from $A_2$ to make up our two element set
$S$ and similarly one element from $B_1$ and one from $B_2$ to make our two element set $T$.  

See the last two sections for the calculation of the arithmetic factor.

\section{Automorphisms}
The final step of this paper is to explain the apparent over-counting that has occurred.
The explanation is that there are automorphisms that have to be taken into account. In this
section we explain these multiplicities.

We start with
\begin{eqnarray*}
\left\{\begin{array}{ll} Nm_1&=Mn_1+h_1\\
Mm_2&=Nn_2+h_2
\end{array}
\right.
\end{eqnarray*}
Suppose $m_1=\mu_1\hat{\mu_1}$, $m_2=\mu_2\hat{\mu_2}$, $n_1=\nu_1\hat{\nu_1}$
and $n_2=\nu_2\hat{\nu_2}$.
Multiply the first equation by $\mu_2\nu_2$ and the second equation by $\mu_1\nu_1$.
Let
$$\tilde M= \nu_1\mu_2 M; \qquad \tilde N=\mu_1\nu_2 N; \qquad \tilde {m_1}=\hat{\mu_1}\mu_2; \qquad
 \tilde {m_2}=\mu_1\hat {\mu_2};\qquad  \tilde{n_1} =\hat{\nu_1}\nu_2 ;\qquad \tilde{n_2}=\nu_1\hat{\nu_2}.
$$
Then we have 
\begin{eqnarray*}
\left\{\begin{array}{ll} \tilde{N}\tilde{m_1}&=\tilde{M}\tilde{n_1}+\tilde{h_1}\\
\tilde{M}\tilde{m_2}&=\tilde{N}\tilde{n_2}+\tilde{h_2}
\end{array}
\right.
\end{eqnarray*}
where 
$$\tilde{h_1}=\mu_2\nu_2 h_1  \qquad \mbox{and} \qquad \tilde{h_2}=\mu_1\nu_1h_2.$$

This scheme provides lots of automorphisms and explains the overcounting we have.

Basically there is one automorphism for each quadruple of divisors of $m_1,m_2,n_1$ and $n_2$.
We have $m=m_1m_2$ and $n=n_1n_2$ where if $|A|=k$ and $|B|=\ell$ then $\tau_A$ is a convolution
of $k$ and $\tau_B$ a convolution of $\ell$ atomic functions.  We can think of 
$$A=\{\alpha_1,\dots ,\alpha_k\} \qquad B=\{\beta_1,\dots,\beta_\ell\}$$
and with $I=\{1,2,\dots,k\}$ and $J=\{1,2,\dots ,\ell\}$ we partition $I=I_1\cup I_2$ and
$J=J_1\cup J_2$. Then in our decompositions $A=A_1\cup A_2$ and $B=B_1\cup B_2$ we have
$A_1=\{\alpha_i:i\in I_1\}$ etc. These correspond to the decompositions $m=m_1m_2$ and $n=n_1n_2$.
If we write $m=\mu_1\dots \mu_k$ and $n=\nu_1\dots \nu_\ell$ then we can put
$m_1=\prod_{i\in I_1}\mu_i$ etc. The number of such decompositions of $A$ or of $m$ is just the number of subsets 
of $A$, i.e. $2^k$; and the number for $B$ is $2^\ell$. We can associate an automorphism as 
above with each such decomposition. Therefore, there are $2^{k+\ell}$ automorphisms in total. So each term is 
counted with a multiplicity $2^{k+\ell}$. Now let's see that this overcounting is in agreement 
with  the number of ways of producing the term from the recipe with, say, 
$S=\{\alpha_1,\alpha_2\}$ and $T=\{\beta_1,\beta_2\}$.
The term from the recipe will occur whenever we have a decomposition of $A=A_1\cup A_2$ and
$B=B_1\cup B_2$ in which precisely one of $\alpha_1$ and $\alpha_2$ is in $A_1$ and the other in $A_2$ and
similarly for $B$ and the $\beta$s. How many ways are there to do this? If we say that $\alpha_1$ is to be in $A_1$ 
and $\alpha_2$ in $A_2$ then we have $k-2$ other elements to be partitioned into two sets. There are $2^{k-2}$ subsets
and the chosen subset can be assigned to go with $\alpha_1$ or with $\alpha_2$, so we have an extra factor of 2; then
and then another factor of 2 by putting $\alpha_1$ in $A_2$ and $\alpha_2 $ in $A_1$. Therefore, a total of
$2^k$ ways to do this. And $2^\ell$ for the $\beta$s into the $B$s. So, we have the same amount 
of overcounting as there are automorphisms. Taking this into account, we obtain just a single copy
of each term from the recipe. 

\section{Conclusion} We have shown how to obtain an asymptotic formula with power savings
for the mean square of a Dirichlet polynomial of length $X$ where $T^2\ll X\ll T^3$ with coefficients
that are general divisor functions in two different ways: one way is via 
Perron's formula and the recipe, and the other is by calculating a convolution of shifted divisor correlations.
The two approaches give exactly the same answer. 

In the next paper, which will conclude this introductory series, 
we will consider the completely general situation with an arbitrary length Dirichlet polynomial.

\section{The semi-diagonal arithmetic factor}
 
It remains to prove that the arithmetic factors agree.  This calculation is surprisingly involved. In order to carry it out with
minimal notational difficulties we introduce a new set of notation. These appendices are self-contained.

We begin by introducing a little notation.
First of all, we are working locally; basically we are identifying the local $p$-factor in an Euler product.
As far as we are concerned $p$ is fixed for this discussion so we often suppress it. In fact we write $X$ for $1/p$
and mostly consider power series in $X$. We take the unusual step of suppressing not only the prime $p$ 
but the divisor function and so we write $A(n)$ in place of $\tau_A(p^n)$. Also, for a set $A$ we let 
$$A_\alpha=\{a+\alpha:a\in A\}.$$
A further piece of notation: $A^+=A\cup\{0\}$. We have two important identities. The first is 
$$A^+(d)=A(d)+A^+(d-1).$$
This is a special case of
$$ (A\cup\{-\alpha\})(d-1) = X^{\alpha} \bigg( (A\cup\{-\alpha\})(d) - A(d)\bigg).$$
The other identity is 
$$\sum_{r=0}^R A(r+M)=  A^+(R+M)- A^+(M-1)$$
which follows by repeated application of the first identity. 

For arbitrary sets  $A$,$B$, $C$ and $D$  we let
$$\mathcal C(A,B):=\sum_{M=0}^\infty A(M)B(M)X^M$$
and
$$\mathcal F(A,B;C,D)=\sum_{K,L,M}  A(K)B(K+M)C(L)D(L+M)X^{K+L+M};$$

Also, we let 
$$Z(A)=\sum_{j=0}^\infty A(j) X^j=\prod_{a\in A}(1-X^{1+a})^{-1}.$$

We have a lemma about $\mathcal F$ and $\mathcal C$ which is really just a formal manipulation; 
consequently we state it in a more general form. 
\begin{lemma}For any 4 functions $a,A,b,B$ 
let $$F(a,A;b,B)=\sum_{K,L,M} a(K)A(K+M)b(L)B(L+M)X^{K+L+M}$$
and $$C(a,b)=\sum_{r=0}^\infty a(r)b(r)X^r$$
we have
$$F(a,A;b,B)+F(A,a;B,b)=C(A\star b,a\star B)+C(a,A)C(b,B).
$$
\end{lemma}
\begin{proof}
Let $Y=\sqrt{X}e(\theta)$. Then  
\begin{eqnarray*}
F(a,A;b,B)&=&\int_0^1 \sum_{r,s,M,N}a(r)A(r+M) {Y}^{r+s+M} 
b(s)B(s+N)  \overline{Y}^{r+s+N}  ~d\theta\\
&=& 
\int_0^1 \sum_{R,S\atop
r\le R; s\le S} a(r)A(R) {Y}^{s+R} 
b(s)B(S)  \overline{Y}^{r+S}  ~d\theta\\
&=& \sum_{r+S=R+s\atop
r\le R; s\le S}a(r)A(R)  
b(s)B(S) X^{r+S}
\end{eqnarray*}
The latter sum is 
\begin{eqnarray*}&&
 \sum_{r+S=R+s}a(r)A(R)  
b(s)B(S) X^{r+S}
- \sum_{r+S=R+s\atop
r> R; s> S}a(r)A(R)  
b(s)B(S) X^{r+S}\\
&&\qquad = C(a\star B,A\star b)+C(a,A)C(b,B)- F(A,a;B,b) 
\end{eqnarray*}
as desired. 
\end{proof}

Now we address the arithmetic factor from the semi-diagonal term. 
The $p$ part of 
\begin{eqnarray*}
&&  Z((A'_1)_{-\alpha})Z((B'_1)_{-\beta})
\sum_{  (M,N)=1} \frac{1 }{M^{1-\beta}N^{1-\alpha}}
 \sum_{\ell, d} \frac{\tau_{A_2}(N\ell)\tau_{B_2}(M\ell)}
{  \ell^{1+s} d^{1+s } }
\sum_{q\ge1} \frac{\mu(q)(qd,N)^{1-\alpha}(qd,M)^{1- \beta}}{q^{2-\alpha- \beta}}    \\
&&\qquad \qquad \times 
   G_{{A_1}}\left(1- \alpha,\frac {qd}{(qd,N)} \right)  
G_{{B_1}}\left(1-\beta,\frac {qd}{(qd,M)} \right)  
\end{eqnarray*}
is  (after setting $s=0$)
\begin{eqnarray*}
&& 
\sum_{  \min(M,N)=0} X^{M(1-\beta)+N(1-\alpha)}
 \sum_{\ell, d}A_2(N+\ell)B_2(M+\ell)
X^{\ell+d }\\
&&\qquad \times 
\sum_{q }  \mu(p^q)X^{-\min(q+d,N)(1-\alpha)-\min(q+d,M)(1- \beta)+q(2-\alpha- \beta)} 
   \\
&&\qquad \qquad \times 
\sum_{j,k}A_1'(j+q+d-\min(q+d,N))B_1'(k+q+d-\min(q+d,M))X^{j(1-\alpha)+k(1-\beta)}
\end{eqnarray*}
We use 
$$\sum_{\min(M,N)=0}f(M,N)=\sum_M f(M,0)+\sum_N f(0,N)-f(0,0)$$
and get $S_L+S_R-S_0$ where 
\begin{eqnarray*}
&& 
S_0=
 \sum_{\ell, d,q,j,k}A_2(\ell)B_2(\ell)
  \mu(p^q)
A_1'(j+q+d )B_1'(k+q+d) X^{ q(2-\alpha- \beta)+\ell + d +j(1-\alpha)+k(1-\beta)},
\end{eqnarray*}
\begin{eqnarray*}
&& 
S_L=\sum_{ M} X^{M(1-\beta)}
 \sum_{\ell, d}A_2(\ell)B_2(M+\ell)
X^{\ell + d }
\sum_{q }  \mu(p^q)X^{- \min(q+d,M)(1- \beta)+q(2-\alpha- \beta)} 
   \\
&&\qquad \qquad \times 
\sum_{j,k}A_1'(j+q+d )B_1'(k+q+d-\min(q+d,M) )X^{j(1-\alpha)+k(1-\beta)}.
\end{eqnarray*}
\begin{eqnarray*}
&& 
S_R=\sum_{ N} X^{N(1-\alpha)}
 \sum_{\ell, d}A_2(N+\ell)B_2(\ell)
X^{\ell + d }
\sum_{q }  \mu(p^q)X^{-\min(q+d,N)(1-\alpha) +q(2-\alpha- \beta)} 
   \\
&&\qquad \qquad \times 
\sum_{j,k}A_1'(j+q+d-\min(q+d,N))B_1'(k+q+d)X^{j(1-\alpha)+k(1-\beta)}.
\end{eqnarray*}
We expand the $q$ sum in $S_0$ to get
\begin{eqnarray*}
&& 
S_0=
 \sum_{\ell, d,j,k}A_2(\ell)B_2(\ell)
A_1'(j+d )B_1'(k+d) X^{ \ell + d +j(1-\alpha)+k(1-\beta)}\\
&&\qquad
-\sum_{\ell, d,j,k}A_2(\ell)B_2(\ell)
A_1'(j+1+d )B_1'(k+1+d) X^{ 2-\alpha- \beta+\ell + d +j(1-\alpha)+k(1-\beta)}.
\end{eqnarray*}
This telescopes in $j$ and $k$ to give
\begin{eqnarray*}
S_0&=&
 \sum_{\ell, d,j}A_2(\ell)B_2(\ell)
A_1'(j+d )B_1'(d) X^{ \ell + d +j(1-\alpha) }
\\&&+ \sum_{\ell, d,k}A_2(\ell)B_2(\ell)
A_1'(d )B_1'(k+d) X^{ \ell + d +k(1-\beta)}
- \sum_{\ell, d }A_2(\ell)B_2(\ell)
A_1'(d )B_1'(d) X^{ \ell + d  }\\
&=& \mathcal C(A_2,B_2)\bigg(\sum_{r} 
X^{ r(1 -\alpha)} A_1'(r )\sum_{d\le r} X^{d\alpha   } B_1'(d)
+\sum_{r} 
X^{ r(1 -\beta)} B_1'(r )\sum_{d\le r}  X^{d\beta   }  A_1'(d) -\mathcal C(A_1',B_1')\bigg)\\
&=&
 \mathcal C(A_2,B_2)\bigg(\sum_{r} 
X^{ r(1 -\alpha)} A_1'(r )  ((B_1')_{\alpha})^+(r) 
+\sum_{r} 
X^{ r(1 -\beta)} B_1'(r ) (( A_1')_{\beta})^+(r) -\mathcal C(A_1',B_1')\bigg)\\
&=& \mathcal C(A_2,B_2)\bigg( \mathcal C((A_1')_{-\alpha},
  ((B_1')_{\alpha})^+) 
+ \mathcal C(
 ( B_1')_{-\beta},(( A_1')_{\beta})^+)-\mathcal C(A_1',B_1')\bigg).
\end{eqnarray*}
This may be rewritten as 
\begin{eqnarray*}
S_0=
\mathcal C(A_2,B_2)\big( \mathcal C(A_1',
  B_1'\cup\{-\alpha\}) 
+ \mathcal C(
  B_1',A_1'\cup\{-\beta\})-\mathcal C(A_1',B_1')\big).
\end{eqnarray*}

Now we turn to $S_L$. 
 Expanding in $q$ we have
\begin{eqnarray*}
S_L 
&=&\sum_{ M,\ell,d,j,k} 
 A_2(\ell)B_2(M+\ell)
 A_1'(j+d )B_1'(k+d-\min(d,M) )\\&&
\qquad \qquad \qquad \qquad \times X^{\ell+d+M(1-\beta)-\min(d,M)(1- \beta)+j(1-\alpha)+k(1-\beta)}\\&&\qquad \qquad - 
\sum_{ M,\ell,d,j,k}
 A_2(\ell)B_2(M+\ell)A_1'(j+1+d )B_1'(k+1+d-\min(1+d,M) ) \\&&\qquad \qquad \qquad \qquad 
\times X^{\ell + d +M(1-\beta)- \min(1+d,M)(1- \beta)+2-\alpha- \beta+j(1-\alpha)+k(1-\beta)} .
\end{eqnarray*}
We split this into $S_L=S_L^-+S_L^+$ where $S_L^-$ denotes those terms for which  $d<M$ and 
$S_L^+$ contains those terms with $d\ge M$. We have
 \begin{eqnarray*}
S_L^- 
&=&\sum_{ M,\ell,j,k\atop d<M} 
 A_2(\ell)B_2(M+\ell)
 A_1'(j+d )B_1'(k  ) X^{\ell+M(1-\beta)+d  \beta+j(1-\alpha)+k(1-\beta)}\\&& - 
\sum_{ M,\ell,j,k\atop d<M }
 A_2(\ell)B_2(M+\ell)A_1'(j+1+d )B_1'(k  )  
X^{\ell  +M(1-\beta)  + d\beta +(j+1)(1-\alpha)+k(1-\beta)} .
\end{eqnarray*} 
The sum over $j$ telescopes; we are left with 
\begin{eqnarray*}
S_L^- 
&=&\sum_{ M,\ell,k\atop d<M} 
 A_2(\ell)B_2(M+\ell)
 A_1'(d )B_1'(k  ) X^{\ell+M(1-\beta)+d  \beta +k(1-\beta)}
\end{eqnarray*}
 We  execute the sum over $k$ to obtain
\begin{eqnarray*}
S_L^- 
&=&Z((B_1')_{- \beta})\sum_{ M,\ell,\atop d<M} 
 A_2(\ell)B_2(M+\ell)
 A_1'(d ) X^{\ell+M(1-\beta)+d \beta  } .
\end{eqnarray*} 
The sum over $d$ gives 
\begin{eqnarray*}
S_L^- 
&=&Z((B_1')_{- \beta})\sum_{ M,\ell } 
 A_2(\ell)B_2(M+\ell)
 ((A_1')_{\beta})^+(M-1 ) X^{\ell +M(1-\beta)   } \\
&=& 
Z((B_1')_{- \beta})\sum_{ M,\ell } 
 A_2(\ell)B_2(M+\ell)
 \big(((A_1')_{\beta})^+(M)-(A_1')_{\beta}(M)\big) X^{\ell +M(1-\beta)   } \\
&=& Z((B_1')_{- \beta})\sum_{ M,\ell } 
 A_2(\ell)B_2(M+\ell)
 \big((A_1'\cup\{-\beta\})(M)-A_1'(M)\big) X^{\ell +M   }\\
&=&  Z((B_1')_{- \beta})\big(\mathcal C (A_1'\cup A_2\cup\{-\beta\},B_2)-\mathcal C(A_1'\cup A_2,B_2)\big).
\end{eqnarray*} 
Now we consider $S_L^+$.
 We have
\begin{eqnarray*}
S_L^+ 
&=&\sum_{ M,\ell,j,k\atop d\ge M} 
 A_2(\ell)B_2(M+\ell)
 A_1'(j+d )B_1'(k+d-M )  X^{\ell+d +j(1-\alpha)+k(1-\beta)}\\&&  \qquad - 
\sum_{ M,\ell,j,k\atop d\ge M}
 A_2(\ell)B_2(M+\ell)A_1'(j+1+d )B_1'(k+1+d-M )  X^{\ell + d  +2- \alpha- \beta+j(1- \alpha)+k(1- \beta)} .
\end{eqnarray*}
This sum telescopes in $j$ and $k$. We have
\begin{eqnarray*}
S_L^+ 
&=&\sum_{ M,\ell,j\atop d\ge M} 
 A_2(\ell)B_2(M+\ell)
 A_1'(j+d )B_1'(d-M )  X^{\ell+d +j(1-\alpha) }\\&&  \qquad 
+
\sum_{ M,\ell,k\atop d\ge M} 
 A_2(\ell)B_2(M+\ell)
 A_1'(d )B_1'(k+d-M )  X^{\ell+d  +k(1-\beta)}\\&&  \qquad 
-
\sum_{ M,\ell \atop d\ge M} 
 A_2(\ell)B_2(M+\ell)
 A_1'(d )B_1'(d-M )  X^{\ell+d  }.
\end{eqnarray*}
We replace $d$ by $d+M$ and have 
\begin{eqnarray*}
S_L^+ 
&=&\sum_{ M,\ell,j,d} 
 A_2(\ell)B_2(M+\ell)
 A_1'(j+d +M)B_1'(d )  X^{\ell+d+M +j(1-\alpha) }\\&&  \qquad 
+
\sum_{ M,\ell,k,d} 
 A_2(\ell)B_2(M+\ell)
 A_1'(d+M )B_1'(k+d )  X^{\ell+d +M +k(1-\beta)}\\&&  \qquad 
-
\sum_{ M,\ell,d} 
 A_2(\ell)B_2(M+\ell)
 A_1'(d +M)B_1'(d )  X^{\ell+d+M  }.
\end{eqnarray*}
In the first term we replace $j+d$ by $r$ and sum over $d$; it becomes
\begin{eqnarray*}
\sum_{ M,\ell,r} 
 A_2(\ell)B_2(M+\ell)
 (A_1')_{-\alpha}(r +M)((B_1')_{\alpha})^+(r )  X^{\ell+r+M+M\alpha   }.
\end{eqnarray*}
 In the second term we execute the sum over $k$ as follows:
\begin{eqnarray*}
\sum_{k,d} 
 A_1'(d+M )B_1'(k+d )  X^{ d  +k(1-\beta)}&=&
 \sum_{K}(B_1')_{-\beta}(K)\sum_{d\le K}A_1'(d+M) X^{ d  +K\beta}\\
&=& X^{-M\beta}\sum_{K}(B_1')_{-\beta}(K)X^K\sum_{d\le K}(A_1')_{\beta}(d+M)  \\
&=& X^{-M\beta}\sum_{K}(B_1')_{-\beta}(K)X^K\big( ((A_1')_{\beta})^+(K+M)\\
&& \qquad \qquad -((A_1')_{\beta})^+(M-1)\big)
\end{eqnarray*}
This may be rewritten as
\begin{eqnarray*}&&
  X^{-M\beta}\sum_{K}(B_1')_{-\beta}(K)X^K ((A_1')_{\beta})^+(K+M)\\&&\qquad 
-X^{-M\beta} ((A_1')_{\beta})^+(M)Z((B_1')_{-\beta})+X^{-M\beta} (A_1')_{\beta}(M)Z((B_1')_{-\beta}).
\end{eqnarray*}
 Thus, altogether we have
\begin{eqnarray*}
S_L^+ 
&=&\sum_{ M,\ell,r} 
 A_2(\ell)B_2(M+\ell)
 (A_1')_{-\alpha}(r +M)((B_1')_{\alpha})^+(r )  X^{\ell+r+M+M\alpha   }\\&&  \qquad 
+
\sum_{ M,\ell} 
 A_2(\ell)B_2(M+\ell)\bigg(
 X^{-M\beta}\sum_{K}(B_1')_{-\beta}(K)X^K ((A_1')_{\beta})^+(K+M)\\&&\qquad 
-X^{-M\beta} ((A_1')_{\beta})^+(M)Z((B_1')_{-\beta})+X^{-M\beta} (A_1')_{\beta}(M)Z((B_1')_{-\beta})
\bigg) X^{\ell+M}\\&&  \qquad 
-
\sum_{ M,\ell,d} 
 A_2(\ell)B_2(M+\ell)
 A_1'(d +M)B_1'(d )  X^{\ell+d+M  }.
\end{eqnarray*}
In the first line notice that $(((B_1')_{\alpha})^+)_{-\alpha}=B_1'\cup\{-\alpha\}$. Also,
recall our notation:
\begin{eqnarray*}
\mathcal F (A,B;C,D)=\sum_{K,L,M}A(K)B(K+M)C(L)D(L+M)X^{K+L+M}
.
\end{eqnarray*}
Using this notation we have that 
\begin{eqnarray*}
S_L^+ 
&=& \mathcal F(B_1'\cup\{-\alpha\},A_1' ;A_2,B_2)+
\mathcal F (B_1',A_1'\cup\{-\beta\};A_2,B_2)-
\mathcal F(B_1',A_1';A_2,B_2)\\&&\qquad
-Z((B_1')_{-\beta})\mathcal C(A_1'\cup A_2\cup\{-\beta\},B_2)+Z((B_1')_{-\beta})\mathcal C(A_1'\cup A_2,B_2).
\end{eqnarray*}
We add this with our expression for $S_L^-$ and have  
\begin{eqnarray*}
S_L 
&=& \mathcal F(B_1'\cup\{-\alpha\},A_1' ;A_2,B_2)+
\mathcal F (B_1',A_1'\cup\{-\beta\};A_2,B_2)-
\mathcal F(B_1',A_1';A_2,B_2).
\end{eqnarray*} 
The expression for $S_R$ is obtained by the symmetry
 $\alpha \leftrightarrow \beta$;  
$A_1\leftrightarrow B_1$; and $A_2\leftrightarrow B_2$.   Thus,
\begin{eqnarray*}
S_R 
&=& \mathcal F(A_1'\cup\{-\beta\},B_1' ;B_2,A_2)+
\mathcal F (A_1',B_1'\cup\{-\alpha\};B_2,A_2)-
\mathcal F(A_1',B_1';B_2,A_2).
\end{eqnarray*} 
Recall that
$$\mathcal F(A,B;C,D)+\mathcal F(B,A;D,C)=\mathcal C(A\cup D,B\cup C)+\mathcal C(A,B)C(C,D).
$$
Thus, 
\begin{eqnarray*}&&
S_L+S_R=\mathcal C(A_1'\cup A_2\cup\{-\beta\},B_1'\cup B_2)+
C(A_1'\cup A_2 ,B_1'\cup B_2\cup\{-\alpha\})\\
&&\qquad -C(A_1'\cup A_2 ,B_1'\cup B_2)+\mathcal C(A_1'\cup \{-\beta\},B_1')\mathcal C(B_2,A_2)
\\
&&\qquad \qquad
+\mathcal C(A_1' ,B_1'\cup\{-\alpha\})\mathcal C(B_2,A_2)
-\mathcal C(A_1' ,B_1')\mathcal C(B_2,A_2).
\end{eqnarray*}
Adding this to $-S_0$ we have 
\begin{eqnarray*}
S_L+S_R -S_0&=& \mathcal C(A_1'\cup A_2\cup\{-\beta\},B_1'\cup B_2)+
C(A_1'\cup A_2 ,B_1'\cup B_2\cup\{-\alpha\})\\&&\qquad -C(A_1'\cup A_2 ,B_1'\cup B_2).
\end{eqnarray*}
This is equal to 
\begin{eqnarray*}
(1-X^{1-\alpha-\beta}) ~ \mathcal C(A_1'\cup A_2\cup\{-\beta\},B_1'\cup B_2\cup \{-\alpha\})
\end{eqnarray*}
as desired.

\section{Proof of Theorem 2}
We shall it convenient to recast the identity of Theorem 2 using a set-theoretic language.  
\subsection{A reformulation of the identity}
We begin with 4 sets $A,B,C$ and $D$ and 4 numbers $\alpha,\beta,\gamma$ and $\delta$.  
We consider
\begin{eqnarray*}&& \sum_{\min(M,N)=0} X^{-M(\gamma+\beta)-N(\alpha+\delta)} 
\Sigma_1(M,N)\Sigma_2(M,N)X^{M+N}
\end{eqnarray*}
where 
\begin{eqnarray*}
 \Sigma_1(M,N)&=&\sum_{d,j,k\atop q\le 1  }
  (-1)^q X^{d(\alpha+\beta)}{A_{-\alpha}}(j+q+d-\min(q+d,N))\\&&\qquad \times {B_{-\beta}}(k+q+d-\min(q+d,M))  
  X^{2q+d+j+k-\min(q+d,M)-\min(q+d,N)}
\end{eqnarray*} 
and 
\begin{eqnarray*}
 \Sigma_2(M,N)&=&\sum_{d,j,k\atop q\le 1  }
  (-1)^q X^{d(\gamma+\delta)}{C_{-\gamma}}(j+q+d-\min(q+d,M))\\&&\qquad \times {D_{-\delta}}(k+q+d-\min(q+d,N))  
  X^{2q+d+j+k-\min(q+d,M)-\min(q+d,N)}.
\end{eqnarray*} 
The problem is to express this quantity in terms of the $\mathcal C$   function, namely we want to prove that the above is 
$$=(1-X^{1-\alpha-\beta})(1-X^{1-\gamma-\delta})\mathcal C(A\cup C\cup\{-\beta,-\delta\},
B\cup D \cup \{-\alpha,-\gamma\}).$$

\subsection{Initial reductions}

We can decompose the sum over $M$ and $N$ via
 $$\sum_{\min(M,N)=0}f(M,N)=\sum_{M=0}^\infty f(M,0)+\sum_{N=0}^\infty f(0,N)-f(0,0).$$
Thus, 
the sum above is $S_L+S_R-S_0$ where 
\begin{eqnarray*}
S_L= \sum_{M=0}^\infty 
X^{-M(\gamma+\beta) } 
\Sigma_1(M,0)\Sigma_2(M,0)X^{M}
\end{eqnarray*}
\begin{eqnarray*}&&S_R=\sum_{N=0}^\infty 
 X^{ -N(\alpha+\delta)} 
\Sigma_1(0,N)\Sigma_2(0,N)X^{N} 
\end{eqnarray*}
and
 \begin{eqnarray*}
S_0= \Sigma_1(0,0)\Sigma_2(0,0).
\end{eqnarray*} 
 We have
\begin{eqnarray*}&&
S_0=
 \bigg(\sum_{d,j,k\atop q\le 1  }
  (-1)^q X^{d(\alpha+\beta)}{A_{-\alpha}}(j+q+d )  {B_{-\beta}}(k+q+d )  
  X^{2q+d+j+k} \bigg)\\&&\qquad \times 
\bigg(
 \sum_{d,j,k\atop q\le 1  }
  (-1)^q X^{d(\gamma+\delta)}{C_{-\gamma}}(j+q+d ) {D_{-\delta}}(k+q+d )  
  X^{2q+d+j+k }\bigg).
\end{eqnarray*}
The first factor here is 
\begin{eqnarray*}&&
\sum_{d,j,k  }
    X^{d(\alpha+\beta)}{A_{-\alpha}}(j+d )  {B_{-\beta}}(k+d )  
  X^{d+j+k}\\&&\qquad 
-
\sum_{d,j,k   }
    X^{d(\alpha+\beta)}{A_{-\alpha}}(j+1+d )  {B_{-\beta}}(k+1+d )  
  X^{2+d+j+k}
\end{eqnarray*}
which telescopes in $j$ and $k$. Thus, it is 
\begin{eqnarray*}&&
\sum_{d,j }
    X^{d(\alpha+\beta)}{A_{-\alpha}}(j+d )  {B_{-\beta}}(d )  
  X^{d+j}
+\sum_{d,k  }
    X^{d(\alpha+\beta)}{A_{-\alpha}}(d )  {B_{-\beta}}(k+d )  
  X^{d+k}\\
&&\qquad 
-\sum_{d  }
    X^{d(\alpha+\beta)}{A_{-\alpha}}(d )  {B_{-\beta}}(d )  
  X^{d}.
\end{eqnarray*}
The first term here is 
\begin{eqnarray*}&&
\sum_{d,j }
    X^{d(\alpha+\beta)}{A_{-\alpha}}(j+d )  {B_{-\beta}}(d )  
  X^{d+j}
=\sum_{J=0}^\infty {A_{-\alpha}}(J )X^J \sum_{d\le J}  {B_{-\beta}}(d ) X^{d(\alpha+\beta)}.
\end{eqnarray*}  
Now 
\begin{eqnarray*}
{B_{-\beta}}(d ) X^{d(\alpha+\beta)}={B_{\alpha}}(d )
\end{eqnarray*}
and 
\begin{eqnarray*}
\sum_{d\le J}{B_{\alpha}}(d )
=(B_{\alpha})^+(J ).
 \end{eqnarray*}
Thus, the above is 
\begin{eqnarray*}
\sum_{J=0}^\infty 
{A_{-\alpha}}(J )(B_{\alpha})^+(J )X^J
=\mathcal C(A_{-\alpha},(B_{\alpha})^+)=\mathcal C(A ,B\cup\{-\alpha\}).
\end{eqnarray*}
The second term is 
\begin{eqnarray*}
\mathcal C((A_{\beta})^+,B_{-\beta})=\mathcal C(A\cup\{-\beta\},B )
\end{eqnarray*}
and the third term is 
\begin{eqnarray*}
\mathcal C(A ,B)
.
\end{eqnarray*}
We can do the same with the second factor. The net result is that  
\begin{eqnarray*}
S_0&=&\bigg(\mathcal C(A ,B\cup\{-\alpha\}) +
\mathcal C(A\cup \{-\beta\},B )-\mathcal C(A,B)\bigg)\\&&
\qquad \times 
\bigg(\mathcal C(C ,D\cup\{-\gamma\}) +
\mathcal C(C\cup\{-\delta\},D )-\mathcal C(C,D)\bigg).
\end{eqnarray*}

Next we analyze $S_L$. First consider $\Sigma_1(M,0)$:
\begin{eqnarray*}&&
 \Sigma_1(M,0)=\sum_{d,j,k  }
    X^{d(\alpha+\beta)}{A_{-\alpha}}(j+d ) {B_{-\beta}}(k+d-\min(d,M))  
  X^{d+j+k-\min(d,M) }\\ && \qquad 
-
\sum_{d,j,k  }
    X^{d(\alpha+\beta)}{A_{-\alpha}}(j+1+d ) {B_{-\beta}}(k+1+d-\min(1+d,M))  
  X^{2+d+j+k-\min(1+d,M) }.
\end{eqnarray*} 
We split this into the terms with $d< M$ and those with $d\ge M$. We have
\begin{eqnarray*}
\Sigma_1^-(M,0)
&=&\sum_{j,k\atop 
d< M  }
    X^{d(\alpha+\beta)}{A_{-\alpha}}(j+d ) {B_{-\beta}}(k )  
  X^{j+k }\\ && \qquad 
-
\sum_{j,k\atop d < M  }
    X^{d(\alpha+\beta)}{A_{-\alpha}}(j+1+d ) {B_{-\beta}}(k )  
  X^{1+j+k  }
\\&=& Z(B_{-\beta})\bigg(
\sum_{j\atop 
d < M  }
    X^{d(\alpha+\beta)}{A_{-\alpha}}(j+d )  
  X^{j } 
-
\sum_{j\atop d < M  }
    X^{d(\alpha+\beta)}{A_{-\alpha}}(j+1+d )  
  X^{1+j }\bigg).
\end{eqnarray*}
The sum over $j$ telescopes so that this is 
\begin{eqnarray*}
\Sigma_1^-(M,0)
&=& Z(B_{-\beta})
\sum_{ 
d < M  }
    X^{d(\alpha+\beta)}{A_{-\alpha}}(d ) \\
&=& Z(B_{-\beta})
\sum_{ 
d < M  }
     {A_{\beta}}(d ) =  Z(B_{-\beta})(A_{\beta})^+(M-1).
\end{eqnarray*}

Next we consider 
\begin{eqnarray*}
\Sigma_1^+(M)&=&
 \sum_{j,k\atop d \ge M  }
    X^{d(\alpha+\beta)}{A_{-\alpha}}(j+d ) {B_{-\beta}}(k+d-M)  
  X^{d+j+k-M }\\ && \qquad 
-
\sum_{j,k\atop d \ge M  }
    X^{d(\alpha+\beta)}{A_{-\alpha}}(j+1+d ) {B_{-\beta}}(k+1+d-M)  
  X^{2+d+j+k-M }.
\end{eqnarray*}
We replace $d$ by $d+M$ and have 
\begin{eqnarray*}
\Sigma_1^+(M)&=&
 \sum_{j,k,d  }
    X^{(d+M)(\alpha+\beta)}{A_{-\alpha}}(j+d+M ) {B_{-\beta}}(k+d)  
  X^{d+j+k }\\ && \qquad 
-
\sum_{j,k,d  }
    X^{(d+M)(\alpha+\beta)}{A_{-\alpha}}(j+1+d+M ) {B_{-\beta}}(k+1+d)  
  X^{2+d+j+k }.
\end{eqnarray*}
Now the sum over $j$ and $k$ telescopes and we have 
\begin{eqnarray*}
\Sigma_1^+(M)&=&
 \sum_{j,d  }
    X^{(d+M)(\alpha+\beta)}{A_{-\alpha}}(j+d+M ) {B_{-\beta}}(d)  
  X^{d+j }\\ && \qquad 
+ \sum_{k,d  }
    X^{(d+M)(\alpha+\beta)}{A_{-\alpha}}(d+M ) {B_{-\beta}}(k+d)  
  X^{d+k }\\ && \qquad 
- \sum_{d  }
    X^{(d+M)(\alpha+\beta)}{A_{-\alpha}}(d+M ) {B_{-\beta}}(d)  
  X^{d }\\ && \qquad 
\end{eqnarray*}
 
We recognize a convolution in the first term and rewrite this as
\begin{eqnarray*}
\Sigma_1^+(M)&=&
 \sum_{r }
    X^{M(\alpha+\beta)}{A_{-\alpha}}(r+M ) (B_{\alpha})^+(r)  
  X^{r }\\ && \qquad 
+ \sum_{k,d  }
    X^{(d+M)(\alpha+\beta)}{A_{-\alpha}}(d+M ) {B_{-\beta}}(k+d)  
  X^{d+k }\\ && \qquad 
- \sum_{d  }
    X^{(d+M)(\alpha+\beta)}{A_{-\alpha}}(d+M ) {B_{-\beta}}(d)  
  X^{d }\\ && \qquad 
\end{eqnarray*}
 The middle term here may be written as   
\begin{eqnarray*}
  \sum_{K  }
     {B_{-\beta}}(K)  
  X^{K } \sum_{d\le K}A_{\beta}(d+M )&=& \sum_{K  }
     {B_{-\beta}}(K)  
  X^{K } \left((A_{\beta})^+(K+M )-(A_{\beta})^+(M-1 )\right)\\
&=& \sum_K  {B_{-\beta}}(K)  
  (A_{\beta})^+(K+M )X^K-Z(B_{-\beta})(A_{\beta})^+(M-1 ).
\end{eqnarray*}
The second term of this cancels with $\Sigma_1^-(M,0)$
and so we have 
\begin{eqnarray*}
\Sigma_1(M,0)&=&
    X^{M(\alpha+\beta)}\sum_{K }(B_{\alpha})^+(K) {A_{-\alpha}}(K+M )  
  X^{K }
- X^{M(\alpha+\beta)}\sum_{K }B_{\alpha}(K) {A_{-\alpha}}(K+M )  
  X^{K }\\
&&\qquad 
+\sum_K  {B_{-\beta}}(K)  
  (A_{\beta})^+(K+M )X^K.
\end{eqnarray*}
This may be rewritten as 
\begin{eqnarray*}
\Sigma_1(M,0)&=&X^{M\beta} \bigg(
     \sum_{K }(B\cup\{-\alpha\})(K) A (K+M )  
  X^{K }
-  \sum_{K }B (K) {A }(K+M )  
  X^{K }\\
&&\qquad 
+\sum_K  {B }(K)  
  (A\cup\{-\beta\})(K+M )X^K\bigg)
\end{eqnarray*}
By symmetry
\begin{eqnarray*}
\Sigma_2(M,0)&=&X^{M\gamma}\bigg(
    \sum_{L }(C\cup\{-\delta\})(L) D (L+M )  
  X^{L }
-  \sum_{L }C (L)D(L+M )  
  X^{L }\\
&&\qquad +\sum_L C(L)  
  (D\cup\{-\gamma\})(L+M )X^L\bigg).
\end{eqnarray*}
 Recall that we are trying to evaluate
\begin{eqnarray*}
S_L=\sum_{M}X^{M(1-\gamma-\beta) }  \Sigma_1(M,0)\times \Sigma_2(M,0). 
\end{eqnarray*}
If we multiply out the three terms of $\Sigma_1$ by the three terms of $\Sigma_2$ and then sum over $M$ 
we get a total of nine expressions the first of which is
\begin{eqnarray*}&&
\sum_{K,L,M}(B\cup\{-\alpha\})(K) A (K+M )  (C\cup\{-\delta\})(L) D (L+M )  X^{K+L+M}\\
&&\qquad 
=\mathcal F(B\cup\{-\alpha\},A; C\cup\{-\delta\}, D).
\end{eqnarray*}
Thus we now see $S_L$ as a sum of nine terms of $\mathcal F$ at  different arguments
which we encapsulate in the following  table for $S_L$:
\begin{eqnarray*}
\begin{array}{cccccc}
\#&\mbox{sign}&K&K+M&L&L+M\\
\hline
1 &+& B\cup\{-\alpha\}&A& C\cup\{-\delta\}& D\\
2 &-& B\cup\{-\alpha\}&A& C & D\\
3 &+& B\cup\{-\alpha\}&A& C & D\cup\{-\gamma\}\\
4 &-& B &A& C\cup\{-\delta\}& D\\
5 &+& B &A& C & D\\
6 &-& B &A& C & D\cup\{-\gamma\}\\
7 &+& B &A\cup\{-\beta\}& C\cup\{-\delta\}& D\\
8 &-& B &A\cup\{-\beta\}& C & D\\
9 &+& B &A\cup\{-\beta\}& C & D\cup\{-\gamma\}\\
\end{array}
\end{eqnarray*}
Note that $S_R$ is just the same as $S_L$ but with $\alpha \leftrightarrow \beta$; $\gamma\leftrightarrow \delta$;
$A\leftrightarrow B$; and $C\leftrightarrow D$. Thus, we have the table for $S_R$:
\begin{eqnarray*}
\begin{array}{cccccc}
\#&\mbox{sign}&K&K+M&L&L+M\\
\hline
1 &+& A\cup\{-\beta\}&B& D\cup\{-\gamma\}& C\\
2 &-& A\cup\{-\beta\}&B& D & C\\
3 &+& A\cup\{-\beta\}&B& D & C\cup\{-\delta\}\\
4 &-& A &B& D\cup\{-\gamma\}& C\\
5 &+& A &B& D & C\\
6 &-& A &B& D & C\cup\{-\delta\}\\
7 &+& A &B\cup\{-\alpha\}& D\cup\{-\gamma\}& C\\
8 &-& A &B\cup\{-\alpha\}& D & C\\
9 &+& A &B\cup\{-\alpha\}& D & C\cup\{-\delta\}\\
\end{array}
\end{eqnarray*}
Now we pair up line $x$ from $S_L$ with line $10-x$ from $S_R$ and we use the   lemma
to express the sum of the $\mathcal F$-terms as $\mathcal C$'s. 
We have
\begin{eqnarray*}S_L+S_R&=&
\mathcal C( A\cup C \cup\{-\delta\},B\cup D\cup\{-\alpha\})
+\mathcal C( A\cup\{-\beta\}, B)
~ \mathcal C( D\cup\{-\gamma\}, C) \\&&\qquad
- \mathcal C( A\cup C  ,B\cup D\cup\{-\alpha\})- \mathcal C( A\cup\{-\beta\}, B)
~\mathcal C( D , C)\\&&\qquad
+\mathcal C( A\cup C  ,B\cup D\cup\{-\alpha,-\gamma\})
+\mathcal C( A\cup\{-\beta\}, B) ~\mathcal C( D , C\cup\{-\delta\})\\&&\qquad
 - \mathcal C( A\cup C \cup\{-\delta\},B\cup D )
- \mathcal C( A , B) ~\mathcal C( D\cup\{-\gamma\}, C)\\&&\qquad
+\mathcal C( A\cup C ,B\cup D )+ \mathcal C( A, B)~\mathcal C( D , C)\\&&\qquad 
- \mathcal C( A\cup C  ,B\cup D\cup\{-\gamma\})
- \mathcal C(A , B) ~\mathcal C( D ,  C\cup\{-\delta\})\\&&\qquad
+\mathcal C( A\cup C \cup\{-\beta,-\delta\},B\cup D )+
\mathcal C( A , B\cup\{-\alpha\}) ~\mathcal C( D\cup\{-\gamma\}, C)\\&&\qquad 
-\mathcal C( A\cup C \cup\{-\beta\},B\cup D )-
\mathcal C ( A , B\cup\{-\alpha\})~\mathcal C( D , C)\\&&\qquad
+ \mathcal C( A\cup C \cup\{-\beta\},B\cup D\cup\{-\gamma\})
+ \mathcal C( A , B\cup\{-\alpha\})~\mathcal C( D , C\cup\{-\delta\}).
\end{eqnarray*}
When we subtract $S_0$ all of the terms that are products of two $\mathcal C$s cancel:
\begin{eqnarray*}S_L+S_R-S_0&=&
\mathcal C( A\cup C \cup\{-\delta\},B\cup D\cup\{-\alpha\})
- \mathcal C( A\cup C  ,B\cup D\cup\{-\alpha\}) \\&&\qquad
+\mathcal C( A\cup C  ,B\cup D\cup\{-\alpha,-\gamma\})
 - \mathcal C( A\cup C \cup\{-\delta\},B\cup D )
 \\&&\qquad
+\mathcal C( A\cup C ,B\cup D ) 
- \mathcal C( A\cup C  ,B\cup D\cup\{-\gamma\})
 \\&&\qquad
+\mathcal C( A\cup C \cup\{-\beta,-\delta\},B\cup D )  
-\mathcal C( A\cup C \cup\{-\beta\},B\cup D ) \\&&\qquad
+ \mathcal C( A\cup C \cup\{-\beta\},B\cup D\cup\{-\gamma\}).
\end{eqnarray*}

 \subsection{The final reckoning}
A generalization of 
$$A^+(d)=A(d)+A^+(d-1)$$
is 
$$ (A\cup\{-\alpha\})(d-1) = X^{\alpha} \bigg( (A\cup\{-\alpha\})(d) - A(d)\bigg).$$
We apply this to the expression  
\begin{eqnarray*}&&
(1-X^{1-\alpha-\beta})(1-X^{1-\gamma-\delta}) \sum_{r=0}^\infty(A\cup C \cup \{-\beta,-\delta\})(r)
(B\cup D \cup \{-\alpha,-\gamma\})(r)X^r 
\end{eqnarray*}
and after some work find that it is equal to the expression above for 
$S_L+S_R-S_0$.

\end{document}